\documentclass[10pt]{amsart}

\usepackage{geometry,graphicx,amssymb,amsmath,amsbsy,eucal,amsfonts,mathrsfs,amscd,bm,tcolorbox}

\numberwithin{equation}{section}

\allowdisplaybreaks[3]

\newtheorem{theorem}{Theorem}[section]
\newtheorem{lemma}[theorem]{Lemma}

\theoremstyle{definition}

\theoremstyle{remark}
\newtheorem{remark}[theorem]{Remark}

\newcommand{\eps}{\varepsilon}

\newcommand{\dive}{{\ensuremath\mathop{\mathrm{div}\,}}}

\newcommand{\bld}[1]{\boldsymbol{#1}}

\newcommand{\bv}{\bld{v}}
\newcommand{\bw}{\bld{w}}

\newcommand{\bn}{\bld{n}}
\newcommand{\bu}{\bld{u}}
\newcommand{\tbu}{\bld{\tilde{u}}}

\newcommand{\tsig}{\tilde{\sigma}_f}

\newcommand{\bfeta}{{\bm \eta}}
\newcommand{\dbfeta}{\dot{\bm \eta}}

\newcommand{\UU}{\mathcal{U}}

\newcommand{\eu}{\bld{e_{u}}}
\newcommand{\ef}{e_{f}}
\newcommand{\es}{e_{s}}
\newcommand{\eeta}{\bld{e}_{\eta}}
\newcommand{\edeta}{\dot{\bld{e}}_{\eta}}
\newcommand{\teu}{\tilde{\bld{e}}_{u}}
\newcommand{\tef}{\tilde{e}_{f}}

\title[Robin-Robin Coupling]{Stability and error analysis of a splitting method using Robin-Robin coupling applied to a fluid-structure interaction problem}

\author{Erik Burman, Rebecca Durst, Johnny Guzm\'an}
\begin{document}
\maketitle
\begin{abstract}
We analyze a splitting method for a canonical fluid structure interaction problem. The splittling method uses a Robin-Robin boundary condition, explicit strategy. We prove the method is stable and, furthermore, we provide an error estimate that shows the error at the final time $T$ is $O(\sqrt{T \Delta t})$ where $\Delta t$ is the time step.
\end{abstract}


\section{Introduction}
In this work we are interested in the stability analysis of a loosely
coupled scheme for the approximation of the interaction of a viscous
fluid and an elastic solid. In a loosely coupled (or explicit) scheme the two systems are
solved separately in a staggered manner, passing interface data from one system to
the other between the solves.
It is well known that loosely coupled schemes for fluid structure
interaction have severe stability problems in situations where the
density ratio between the two phases is close to one. This is due to what is known as the added mass effect \cite{causin2005added}. There
has been intense research on approaches that allow for a partial or even
complete decoupling of the two systems without loss of stability,
however very few fully decoupled approaches have been developed with a
satisfactory theoretical foundation. 

A
first step in the direction of decoupling the two systems is the semi-implicit copling schemes \cite{Fernandezetal2007,QuainiQuarteroni2007,BadiaQuainiQuarteroni2008b,BukacCanicGlowMuhaQuaini2014},
where the implicit part of the coupling, typically the elasticity
system and the pressure velocity coupling in the fluid, guarantees
stability, and the explicit step (transport in the fluid) reduces the
computational cost. Such splitting methods nevertheless retain an
implicit part of the same size as the original problem, although the
equations are simplified. Fully explicit coupling was first achieved
by Burman and Fern\`andez \cite{burman2009stabilization} using a formulation based on Nitsche's method,
drawing on an earlier, fully implicit formulation by Hansbo et al.
\cite{Hansbo2005}. Stability was achieved by the addition of a
pressure stabilization that relaxed incompressibility in the vicinity
of the interface. Although the proposed scheme was proved to be
stable it suffered from a strong consistency error of order
$O(\tau/h)$ where $\tau$ and $h$ are the time and space discretization
parameters, respectively. The source of this error was the penalty term
of the Nitsche formulation. This led to the need for very small time
steps combined with iterative corrections, for the method to yield sufficiently accurate approximations.
In a further development Burman and Fern\`andez compared the Nitsche
based method with a closely related method using a Robin type splitting
procedure \cite{burman2014explicit}. Robin type domain decomposition
had already been applied for the preconditionning of monolithic fluid
structure interaction problems by Badia et
al. \cite{badia2008fluid}. The loosely coupled scheme based on
Robin type coupling of \cite{burman2014explicit} was proved to be stable, but only with
the addition of the stabilization term on the pressure at the interface and using a weight in the Robin condition scaling similarly as
the penalty term in the Nitsche method. It was however observed numerically that the Robin-Robin coupling method was stable
also without such a pressure stabilizing term. 

It is the objective of the
present paper to revisit the analysis of the Robin-Robin method
without any additional stabilization (what the authors of \cite{burman2014explicit}
referred to as the genuine Robin-Robin method) and prove stability and
error estimates for this method. To make the results
cleaner and more transparent we do not discretize in space. Instead, our splitting scheme solves a fluid and a solid PDE on each time step. Assuming enough regularity of the local time PDEs, we give a rigorous error
analysis that shows the error in a certain energy norm decreases as $O(\sqrt{T \Delta t})$ for
sufficiently smooth solutions, with a parameter $\lambda$ of the Robin condition chosen
$O(1)$. This leads to convergence of the time discrete approximation
independently of the space discretization. This was not the case in
\cite{burman2014explicit}, where as mentioned above the convergence was hampered by the
$h^{-1}$-scaling of the Robin parameter, imposing a very small time
step and iterative correction steps to achieve sufficient
accuracy. Observe that it is likely that the accuracy of the approach
suggested here can be improved using correction steps for moderate
values of the time step, thanks to the absence of the $h^{-1}$ scaling
in the estimate.  We would like to highlight that our estimates grow like $\sqrt{T}$ instead of an exponential growth. We accomplish this by using a technique used by G. Baker \cite{baker1976error}.

Finally we should mention recent papers for the simpler case of interaction between a fluid
and a thin structure that also have rigorous convergence
analysis \cite{Fernandez2013,FernandezLandjuelaVidrascu2015,bukac2016stability}. For the case of thick
solids the paper of \cite{FernandezMullaert2016} seems to be the
first paper with a rigorous error analysis of a thick wall
structure. The method considered in \cite{FernandezMullaert2016} is a Robin-Neumann coupling
that first appeared in \cite{fernandez2015generalized}. There stability
is achieved by handling the inertial effects of the solid in an
implicit coupling with the fluid. This is then combined with
extrapolation to reduce the splitting error. The leading error in
that method for this approach is $O(\Delta t/\sqrt{h})$ which scales like our
error estimates if $\Delta t = O(h)$. It should be mentioned that the constant of their estimates grow exponentially with $T$.

The outline of the present paper is as follows. In section 2 we
introduce the linear model problem. The proposed Robin-Robin loosely
coupled scheme is introduced in section 3 and the stability is
analyzed in section 4. Finally in section 5 we derive the truncation
error of the splitting and use this result together with the stability
estimate to prove the error estimate.

\section{The Model Problem}
\begin{figure}[h]
\includegraphics[width=0.75\linewidth]{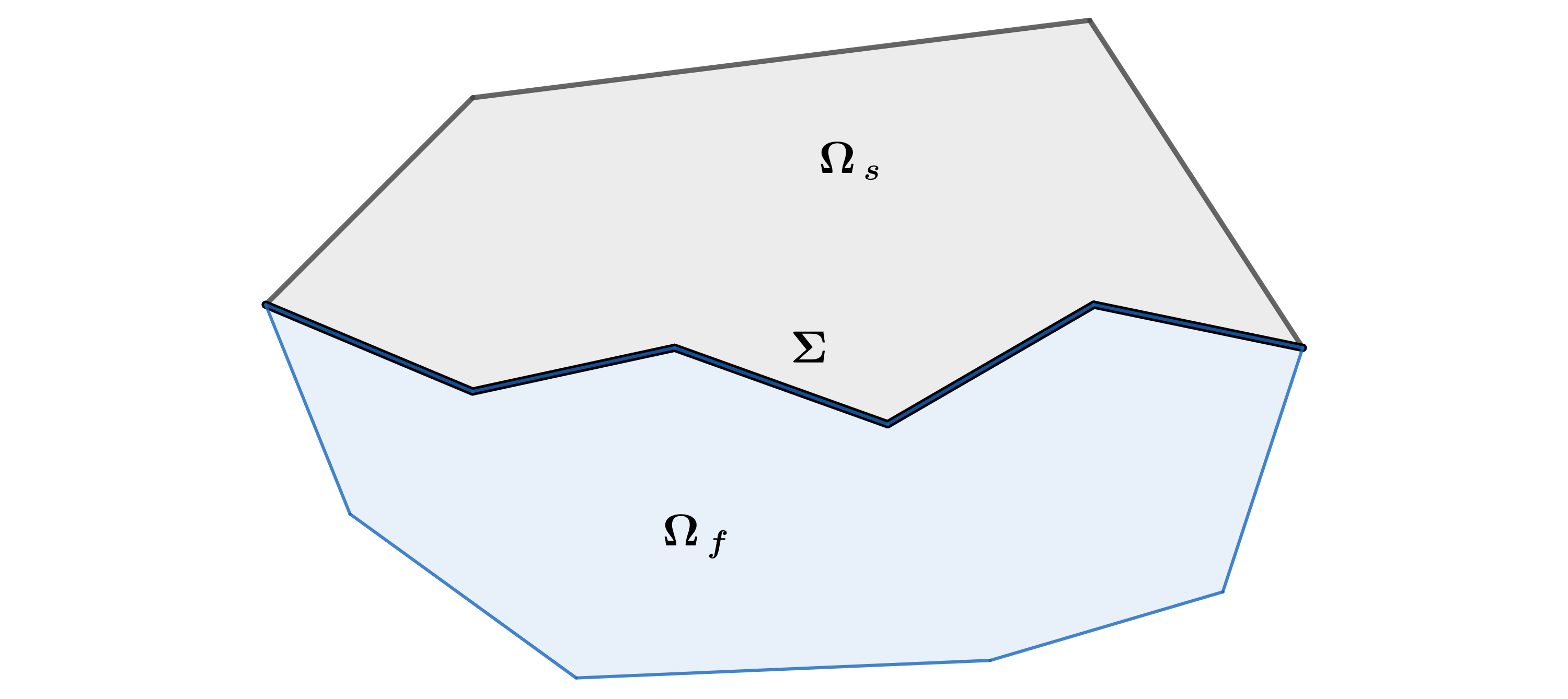}
\caption{An example of the domains $\Omega_s$ and $\Omega_f$ with interface $\Sigma$.\label{fig:fig1}}
\end{figure}
Let $\Omega_s$ and $\Omega_f$ be two domains with a matching interface $\Sigma=\partial \Omega_s \cap \Omega_f$. We also set $\Sigma_i=\partial \Omega_i \backslash \Sigma$ for $i=s,f$. Our fluid is modeled by the Stokes equation on the fluid domain $\Omega_f$.
\begin{equation}\label{Fluid}
\left\{
\begin{alignedat}{2}
\rho_{f} \partial_{t} \bld{\UU}-  \dive \sigma_{\mathcal{F}}  =&   0,   \quad &&  {\rm{in}} \ \Omega_{f} \times (0,T), \\
\dive \bld{\UU}  =&  0,  \quad &&  {\rm{in}} \ \Omega_{f} \times (0,T), \\
\bld{\UU} = &  0,   \quad && {\rm{on}}\ \Sigma_f \times (0,T).
\end{alignedat}
\right.
\end{equation}

Our structure is modeled by the classical linear elasticity equations on the structure domain $\Omega_s$.
\begin{equation}\label{Structure}
\left\{
\begin{alignedat}{2}
\rho_{s} \partial_{tt} \bld{\mathcal{E}}-  \dive \sigma_{\mathcal{S}} = &  0, \quad &&  \ {\rm{in}} \ \Omega_{s} \times (0,T), \\
\bfeta  = &  0,  \quad && {\rm{on}} \ \Sigma_{s}   \times (0,T).
\end{alignedat}
\right.
\end{equation}
Here, $\bld{\mathcal{U}}$ is the velocity of the fluid, $\mathcal{P}$ is the pressure of the fluid, and $\bld{\mathcal{E}}$ is the displacement of the structure. The constants $\rho_f , \rho_s $ are the fluid and solid densities, and $\bn$ and $\bn_s$ represent the outward-facing normal of the fluid and solid domains, respectively. Furthermore, $\sigma_{\mathcal{F}}, \sigma_{\mathcal{S}}$ denote the fluid and solid stress tensors, respectively, and are given by
\begin{alignat*}{1}
\sigma_{\mathcal{F}}= & 2 \mu {\eps} (\bld{\UU}) - \mathcal{P} {\bf I}, \\
\sigma_{\mathcal{S}} = & 2 L_{1} {\eps}(\bld{\mathcal{E}}) + L_{2} (\dive \bld{\mathcal{E}}) {\bf I},
\end{alignat*}
where  ${\eps}$ denotes the symmetric gradient, $\mu$ the viscosity coefficient and $L_1 > 0$, $L_2 \ge 0$ the Lam\'e constants. Then the two problems are coupled via the following interface conditions:
\begin{subequations}\label{Exactinterface}
\begin{alignat}{2}
\bld{\UU} =& \partial_t \bld{\mathcal{E}}  \quad &&{\rm{on}} \  \Sigma, \label{Exactinterface1}\\
\sigma_{\mathcal{S}}\bn_{s} \ + \ \sigma_{\mathcal{F}} \bn  =&  0  \quad &&   {\rm{on}} \ \Sigma. \label{Exactinterface2}
\end{alignat}
\end{subequations}


\section{Splitting Method}

In \cite{burman2014explicit}, several splitting methods were given for the following FSI problem. We will consider one such method. In order to describe it, we consider a uniform grid for the interval $[0,T]$, with step size $\Delta t$.  We assume that there is an integer $N$ so that $N \Delta t=T$ and we let $t_n=\Delta t\, n$. The splitting method sequentially solves the following two sub-problems. The first is the solid problem:

Find $\bfeta^{n+1} $ and $\dbfeta^{n+1}$ such that 
\begin{subequations}\label{solid}
\begin{alignat}{2}
\rho_{s} \partial_{t} {\dbfeta^{n+1}} -  \dive \sigma_{s}^{n+1} = &  0 \quad &&  \ {\rm{in}} \ \Omega_{s} \times [ t_n, t_{n+1}],  \label{solid1} \\
\dbfeta^{n+1}=& \partial_t \bfeta^{n+1}  \quad && \  {{\rm{in}}} \  \Omega_{s} \times [t_{n}, t_{n+1}]   \label{solid2} \\
\sigma_{s}^{n+1}=&  2 L_{1} {\eps}(\bfeta^{n+1}) + L_{2} (\dive \bfeta^{n+1}) {\bf I} \quad && \ {\rm{in}} \ \Omega_{s} \times  [t_n, t_{n+1}]  \label{solid3} \\
\bfeta^{n+1}  = &  0  \quad && \ {\rm{on}} \  \Sigma_s  \times  [t_n, t_{n+1}],  \label{solid4} \\
\lambda \dbfeta^{n+1} + \sigma_{s}^{n+1} \bn_{s}  =&  \lambda \tbu^{n}  -  \tsig^n \bn  \quad &&  \ {\rm{on}} \ \Sigma \times  [t_n, t_{n+1}], \label{solid6}\\ 
\bfeta^{n+1}(\cdot, t_n)= \bfeta^n(\cdot, t_n) , \quad \dbfeta^{n+1}(\cdot, t_n)=& \dbfeta^n(\cdot, t_n)  \quad &&  \ {\rm{on}} \ \Omega_s. \label{solid7} 
\end{alignat}
\end{subequations}
We set, of course, $\bfeta^0(\cdot, t_0)= \bld{\mathcal{E}}(\cdot, t_0)$, $\dbfeta^0(\cdot, t_0)=  \partial_t \bld{\mathcal{E}}(\cdot, t_0)$.  Below we will also set  
$\bu^0(\cdot, t_0)= \bld{\UU}(\cdot, t_0)$. Here  for $n \ge 1$ we set
\begin{equation*}
\tbu^n(x)= \frac{1}{\Delta t}\int_{t_{n-1}}^{t_n} \bu^n(x,s) ds,  \quad \tsig^n(x)=\frac{1}{\Delta t} \int_{t_{n-1}}^{t_n} \sigma_f^n(x,s) ds,
\end{equation*}
and for $n=0$ we set
\begin{equation*}
\tbu^0(x)= \bu^0(x,t_0)  \quad \tsig^0(x)=\sigma_{\mathcal{F}}(x,t_0).
\end{equation*}
Note that $\mathcal{P}(x, t_0)$ is not data and hence, we technically do not know $\sigma_{\mathcal{F}}(x,t_0)$. However, we assume that we have a good approximation of $\mathcal{P}(x, t_0)$. In fact, for simplicity, we will assume that we know $\mathcal{P}(x, t_0)$ exactly. The fluid sub-problems is given in the following:

Find $\bu^{n+1}$ and $p^{n+1}$ such that 
\begin{subequations}\label{fluid}
\begin{alignat}{2}
\rho_{f} \partial_{t} \bu^{n+1}  -  \dive \sigma_{f}^{n+1}  =&   0,   \quad && \  {\rm{in}} \ \Omega_{f} \times [t_n, t_{n+1}] , \label{fluid1} \\
\sigma_{f}^{n+1}= & 2 \mu {\eps} (\bu^{n+1}) - p^{n+1}  {\bf I}, \quad && \ {\rm{in}} \ \Omega_{f} \times [t_n, t_{n+1}],  \label{fluid2} \\ 
\dive \bu^{n+1}  =&  0,  \quad && \ {\rm{in}} \ \Omega_{f} \times [t_n, t_{n+1}],  \label{fluid3} \\
\bu^{n+1} = &  0   \quad && \ {\rm{on}} \  \Sigma_f  \times [t_n, t_{n+1}],  \label{fluid4} \\
\lambda \bu^{n+1} +\sigma_{f}^{n+1} \bn =&  \lambda \dbfeta^{n+1} \ + \tsig^n \bn \quad && \ {\rm{on}} \ \Sigma \times [t_n, t_{n+1}],  \label{fluid6} \\
\bu^{n+1}(\cdot, t_n)=& \bu^n(\cdot, t_n)  \quad &&\ \text{ on } \Omega_f. \label{fluid7}
\end{alignat}
\end{subequations}
We require $\lambda>0$. This strict positivity determines the balancing of the two interface coupling conditions (\ref{Exactinterface1}) and (\ref{Exactinterface2}). A large value of $\lambda$ will emphasize the continuity of velocities and a small value that of stresses. 

Before proceeding, we define the space-time norm on $X$, a Hilbert space
\begin{equation*}
\|v\|_{L^{2}(r_1 ,r_2 ;X)}^2 := \int_{r_1}^{r_2} \|v(s)\|_{X}^{2}ds.
\end{equation*}

\begin{remark}\label{remark}
We will assume that the \eqref{fluid} and \eqref{solid} are well posed and  have enough regularity so that  $\sigma_f^{n+1} \bn \in L^2(t_n, t_{n+1}; \Sigma)$,  $\sigma_s^{n+1} \bn \in L^2(t_n, t_{n+1}; \Sigma)$, $\bu^{n+1} \in L^2(t_n, t_{n+1}; \Sigma)$, $\dbfeta^{n+1} \in L^2(t_n, t_{n+1}; \Sigma)$.  We note, for example,  $\sigma_s^{n+1} \bn \in L^2(t_n, t_{n+1}; \Sigma)$ provided $\bfeta^{n+1} \in L^2(t_n, t_{n+1}; H^{3/2}(\Omega_s))$. We have not been able to find exactly these regularity results in literature nor have we been able to prove them, however, we have found some encouraging results in the literature (see for example \cite{taira2013mixed}, \cite{hayashida1971mixed}). 
\end{remark}

Throughout this paper we will assume that the regularity mentioned in the above remark holds.


\section{Stability Analysis}\label{sec:stability}
We will now prove stability of the splitting method introduced in the last section.  We start with a few preliminary results. The following identity easily follows:
\begin{equation}\label{eq131}
\int_{\Sigma} (\bld{v}-\bld{w}) \cdot \bld{\psi}= \frac{1}{2} \left(\|\bld{v}\|_{L^2(\Sigma)}^2-\|\bld{w}\|_{L^2(\Sigma)}^2+   \|\bld{\psi}-\bld{w}\|_{L^2(\Sigma)}^2-\|\bld{\psi}-\bld{v}\|_{L^2(\Sigma)}^2\right).
\end{equation}
Additionally, if we set $\tilde{\bld{w}}(x)= \frac{1}{\Delta t} \int_{t_{n-1}}^{t_n} \bld{w}(x,s) ds$ then we see that
\begin{alignat}{1}
\nonumber \int_{t_{n}}^{t_{n+1}} \|\tilde{\bld{w}} \|_{L^2(\Sigma)}^2= &\Delta t \int_{\Sigma} (\tilde{\bld{w}}(x))^2 =   \frac{1}{\Delta t} \int_{\Sigma} (\int_{t_{n-1}}^{t_n} \bld{w}(x,s) ds)^2 \\
 \le &   \int_{t_{n-1}}^{t_n} \int_{\Sigma} (\bld{w}(x,s))^2 \, ds   =  \int_{t_{n-1}}^{t_n}  \|\bld{w}(s)\|_{L^2(\Sigma)}^2 ds. \label{eq111}
\end{alignat}

For the analysis that follows, we define the bilinear form $a_{s}(\bw,\bv)$ to be
\begin{equation*}
a_s(\bw, \bv):=2L_1(\eps(\bw), \eps(\bv))_s+ L_2 (\dive \bw, \dive \bv)_s.
\end{equation*}
It induces the norm:
\begin{equation*}
\|\bw\|_S^2:= 2 L_1 \|\eps(\bw)\|_{L^2 (\Omega_s)}^2 + L_2 \|\dive \bw\|_{L^2 (\Omega_s)}^2.
\end{equation*}
Finally, the following quantities willl allow us to state our stability estimates

\begin{alignat*}{2}
\mathbf{E}^{n} \ &:= \ \frac{\rho_{f}}{2} \|\bu^n(t_n)\|^{2}_{L^{2}(\Omega_{f})}  + \frac{\rho_{s}}{2}\|\dbfeta^n(t_n)\|^{2}_{L^{2}(\Omega_{s})}+ \frac{1}{2} \|\bfeta^n(t_n)\|_S^2,  \\
 \mathbf{T}^{n} \ &:= \ 2 \mu \int_{t_{n-1}}^{t_{n}} \| {\eps}(\bu^{n})(s) \|^{2}_{L^{2}(\Omega_{f})}ds + \frac{\lambda}{2} \int_{t_{n-1}}^{t_n}  \|(\dbfeta^{n}-\tbu^{n-1})(s)\|_{L^2(\Sigma)}^2 ds,  \\
\mathbf{S}^{n} \ &:= \ \frac{1}{2\lambda}\int_{t_{n-1}}^{t_{n}} \|\sigma_f^n(s)\bn\|^{2}_{L^{2}(\Sigma)}ds + \frac{\lambda}{2}\int_{t_{n-1}}^{t_{n}}\|\bu^{n}(s)\|^{2}_{L^{2}(\Sigma)}ds,  \qquad \text{ for } n \ge 1, \\
\mathbf{S}^{0}:&= \frac{\Delta t}{2\lambda} \|\sigma_{\mathcal{F}}(t_0)\bn\|^{2}_{L^{2}(\Sigma)}  \ + \ \frac{\lambda \Delta t }{2} \|\bld{\mathcal{U}}(t_0)\|^{2}_{L^{2}(\Sigma)}.
\end{alignat*}

The stability result is given in the following theorem. 
\begin{theorem}
Let $\lambda>0$ and suppose that $\bfeta^{n+1}$ solves \eqref{solid} and $\bu^{n+1}, p^{n+1}$ solve \eqref{fluid} for $0 \le n \le N-1$. Then we have   
\begin{alignat*}{1}
\mathbf{E}^{N}+ \sum_{n=1}^N   \mathbf{T}^{n}+ \mathbf{S}^{N} \le  \mathbf{E}^{0}+ \mathbf{S}^{0}.
\end{alignat*}
\end{theorem}

\begin{proof}

We multiply the first equations of \eqref{solid} and \eqref{fluid} by $\dot{\bfeta}^{n+1}$ and $\bu^{n+1}$, respectively,  integrate, and add the results to get
\begin{alignat}{1}
 \frac{\rho_{f}}{2}\partial_{t} \|\bu^{n+1}\|^{2}_{L^{2}(\Omega_{f})} \ + \ 2 \mu \| {\eps}(\bu^{n+1}) \|^{2}_{L^{2}(\Omega_{f})} \ + \ \frac{\rho_{s}}{2} \partial_{t}\|\dot{\bfeta}^{n+1}\|^{2}_{L^{2}(\Omega_{s})} \ + \ \frac{1}{2}\partial_{t}    \|\bfeta^{n+1} \|_S^2 = J, \label{r43}
\end{alignat}
where
\begin{equation*}
J := \int_{\Sigma}\sigma_{f}^{n+1}\bn \cdot \bu^{n+1} \  + \ \int_{\Sigma}\sigma_{s}^{n+1} \bn_{s} \cdot \dot{\bfeta}^{n+1}.
\end{equation*}

We can then write
\begin{alignat*}{1}
  J=\int_{\Sigma}\sigma_f^{n+1}\bn \cdot (\bu^{n+1}-\dbfeta^{n+1}) \ + \ \int_{\Sigma}(\sigma_s^{n+1} \bn_{s} +\sigma_f^{n+1} \bn)  \cdot \dot{\bfeta}^{n+1}.
\end{alignat*}

From \eqref{solid6} and \eqref{fluid6} we get
\begin{alignat*}{1}
\sigma_s^{n+1} \bn_{s} +\sigma_f^{n+1} \bn= & \lambda (\tbu^n-\bu^{n+1}), \\
\bu^{n+1}-\dbfeta^{n+1}=& \frac{1}{\lambda}(\tsig^n\bn -\sigma_f^{n+1}\bn).
\end{alignat*}
Thus,
\begin{equation}
J\ = \  \frac{1}{\lambda} \int_{\Sigma}\sigma_f^{n+1}\bn \cdot (\tsig^n\bn -\sigma_f^{n+1}\bn)+ \ \lambda  \int_{\Sigma} (\tbu^n-\bu^{n+1}) \cdot \dot{\bfeta}^{n+1}.
\end{equation}
By the relation \eqref{eq131} and the fact that $\frac{1}{\lambda} \|\tsig^n\bn-\sigma_f^{n+1} \bn\|_{L^{2}(\Sigma)}^2= \lambda  \|\dbfeta^{n+1}-\bu^{n+1}\|_{L^2(\Sigma)}^2$,  we obtain
\begin{equation*}
J= \frac{\lambda}{2} \left(\|\tbu^{n}\|_{L^2(\Sigma)}^2-\|\bu^{n+1}\|_{L^2(\Sigma)}^2\right) + \frac{1}{2\lambda} \left(\|\tsig^n\bn\|_{L^{2}(\Sigma)}^2- \|\sigma_{f}^{n+1}\bn\|_{L^{2}(\Sigma)}^2\right) -\frac{\lambda}{2}\|\dbfeta^{n+1}-\tbu^n\|_{L^2(\Sigma)}^2.
\end{equation*}
If we plug this into \eqref{r43} we arrive at
\begin{align*}
\nonumber \frac{\rho_{f}}{2}\partial_{t} &\|\bu^{n+1}\|^{2}_{L^{2}(\Omega_{f})} \ + \ 2 \mu \| {\eps}(\bu^{n+1}) \|^{2}_{L^{2}(\Omega_{f})} \ + \ \frac{\rho_{s}}{2} \partial_{t}\|\dot{\bfeta}^{n+1}\|^{2}_{L^{2}(\Omega_{s})} \ + \ \frac{1}{2}\partial_{t}  \| \bfeta^{n+1} \|_S^2  \\ 
\ & + \  \frac{1}{2\lambda}\|\sigma_{f}^{n+1}\bn\|^{2}_{L^{2}(\Sigma)}  + \  \frac{\lambda}{2} \|\bu^{n+1}\|^{2}_{L^{2}(\Sigma)}  \ =  \frac{1}{2\lambda}\|\tsig^n\bn\|^{2}_{L^{2}(\Sigma)}  + \  \frac{\lambda}{2} \|\tbu^n\|^{2}_{L^{2}(\Sigma)}    -\frac{\lambda}{2}\|\dbfeta^{n+1}-\tbu^n\|_{L^2(\Sigma)}^2.
\end{align*}
After integrating on $[t_n, t_{n+1}]$  and using \eqref{eq111}, we obtain
\begin{alignat*}{1}
\mathbf{E}^{n+1}+ \mathbf{T}^{n+1} + \mathbf{S}^{n+1} \le \mathbf{E}^{n}+ \mathbf{S}^{n}.
\end{alignat*}
The result now follows after summing the above inequalities over all $n$ from $0$ to $N-1$. 
\end{proof}


\section{Error Estimates}\label{sec:error}

We now show that the splitting method with Robin-Robin type boundary conditions described above is, in fact, weakly consistent. In fact, we will prove that the error is $\sqrt{T\Delta t}$. Consider the solutions $\bld{\UU}, \mathcal{P}, \sigma_{\mathcal{F}}, \bld{ \mathcal{E}}, \sigma_{\mathcal{S}}$ of \eqref{Fluid}, \eqref{Structure} and \eqref{Exactinterface}. We use the notation $\bld{\UU}^{n+1}(t,x)=\bld{\UU}(t,x)$ for  $t_n \le t \le t_{n+1} $ and $x \in \Omega$; this similarly holds for the other variables. We then set the errors:
\begin{alignat*}{1}
\eu^{n} \ &= \ \bld{\mathcal{U}}^{n} \ - \ \bu^{n}, \quad \ef^{n} \ = \ \sigma_{\mathcal{F}}^{n} \ - \ \sigma_{f}^{n} \\
\es^{n} \ &= \ \sigma_{\mathcal{S}}^{n} \ - \ \sigma_{s}^{n}, \quad \eeta^{n} \ = \ \bld{\mathcal{E}}^{n} \ - \ \bfeta^{n}, \quad \edeta^{n} \ = \ \dot{\bld{\mathcal{E}}}^{n} \ - \ \dot{\bfeta}^{n}
\end{alignat*}
We also define the following quantities which will be useful to describe our error estimates:
\begin{alignat*}{2}
\mathbb{E}^{n} \ &:= \ \frac{\rho_{f}}{2} \|\eu^{n}(t_{n})\|^{2}_{L^{2}(\Omega_{f})}  + \frac{\rho_{s}}{2}\|\edeta^{n}(t_{n})\|^{2}_{L^{2}(\Omega_{s})}+ \frac{1}{2} \|\eeta^{n}(t_{n})\|_S^2, \\
 \mathbb{T}^{n} \ &:= \ 2 \mu \int_{t_{n-1}}^{t_{n}} \| {\eps}(\eu^{n}(s)) \|^{2}_{L^{2}(\Omega_{f})}ds+  \frac{\lambda}{4}  \int_{t_{n-1}}^{t_{n}} \|\edeta^{n} - \teu^{n-1}\|^{2}_{L^{2}(\Sigma)}  \\
\mathbb{S}^{n} \ &:= \ \frac{1}{2\lambda}\int_{t_{n-1}}^{t_{n}} \|\ef^{n}(s)\bn\|^{2}_{L^{2}(\Sigma)}ds + \frac{\lambda}{2}\int_{t_{n-1}}^{t_{n}}\|\eu^{n}(s)\|^{2}_{L^{2}(\Sigma)}ds, \qquad \text{ for  } n \ge 1, \\
\mathbb{S}^{0} \ &:= \ \frac{\Delta t}{2\lambda} \|\ef^{0}(t_0)\bn\|^{2}_{L^{2}(\Sigma)} + \frac{\lambda \Delta t }{2}\|\eu^{0}(t_0)\|^{2}_{L^{2}(\Sigma)}.
\end{alignat*}
We note that  $\mathbb{E}^{0}=0$ and $\mathbb{S}^{0}=0$. 

For the proof of the error estimates, we will make use of the following lemma. We define
\begin{alignat}{1}
g_3^{n+1}&:=\lambda(\bld{\mathcal{U}}^{n+1}- \tilde{\bld{\mathcal{U}}}^{n}), \label{g3} \\
g_2^{n+1}&:= (\sigma_{\mathcal{F}}^{n+1}\bn \ - \ \tilde{\sigma}_{\mathcal{F}}^{n} \bn). \label{g2}
\end{alignat}
\begin{lemma}\label{lem:g-bound}
For $\bld{\mathcal{U}}^{n}$ and $\sigma_{\mathcal{F}}^{n}$ defined above, we have for $n \ge 1$
\begin{alignat}{1}
\int_{t_{n}}^{t_{n+1}} \|g_3^{n+1}(s)\|^{2}_{L^{2}(\Sigma)}ds \le &  C \lambda^2 (\Delta t)^2 \int_{t_{n-1} }^{ t_{n+1}}  \| \partial_t \bld{\mathcal{U}}(s)\|_{L^2(\Sigma)}^2 ds, \label{lem123}\\
\int_{t_{n}}^{t_{n+1}} \|g_2^{n+1}(s)\|^{2}_{L^{2}(\Sigma)}ds \le &  C(\Delta t)^2 \int_{t_{n-1} }^{ t_{n+1}}  \| \partial_t \sigma_{\mathcal{F}}(s)\bn\|_{L^2(\Sigma)}^2 ds. \label{lem1230}
\end{alignat}
For $n=0$ we have 
\begin{alignat}{1}
 \int_{t_0}^{t_1} \|g_3^{1}(s)\|^{2}_{L^{2}(\Sigma)}ds \le &  C \lambda^2 (\Delta t)^2 \int_{t_0 }^{ t_1}  \| \partial_t \bld{\mathcal{U}}(s)\|_{L^2(\Sigma)}^2 ds, \label{lem123-0} \\
\int_{t_0}^{t_1} \|g_2^{1}(s)\|^{2}_{L^{2}(\Sigma)}ds \le &  C(\Delta t)^2 \int_{t_0 }^{ t_1}  \| \partial_t \sigma_{\mathcal{F}}(s)\bn\|_{L^2(\Sigma)}^2 ds.  \label{lem1230-0}
\end{alignat}

\end{lemma}

\begin{proof}
We only prove \eqref{lem123}  as the proof of the other estimates are similar. We have
\begin{alignat*}{1}
\bld{\mathcal{U}}^{n+1}(s)-\tilde{\bld{\mathcal{U}}}^n(s)=& \frac{1}{\Delta t} \int_{t_{n-1}}^{t_n} (\bld{\mathcal{U}}(s)- \bld{\mathcal{U}}(r)) \, dr= \frac{1}{\Delta t} \int_{t_{n-1}}^{t_n} \int_{r}^s \partial_t \bld{\mathcal{U}}(\theta)  \,d\theta \, dr.
\end{alignat*}
Hence, 
\begin{alignat*}{1}
\int_{t_{n}}^{t_{n+1}} \|\lambda(\bld{\mathcal{U}}^{n+1}(s)- \tilde{\bld{\mathcal{U}}}^{n}(s))\|^{2}_{L^{2}(\Sigma)}ds=& \lambda^2 \int_{t_{n}}^{t_{n+1}} \int_{\Sigma}  (\frac{1}{\Delta t} \int_{t_{n-1}}^{t_n} \int_{r}^s \partial_t \bld{\mathcal{U}}(\theta)  \,d\theta \, dr)^2 \,ds  \\
\le & 2\lambda^2  \int_{t_{n}}^{t_{n+1}} \int_{\Sigma}  \int_{t_{n-1}}^{t_n} \int_{r}^s (\partial_t \bld{\mathcal{U}}(\theta))^2  \,d\theta \, dr \,\,ds  \\ 
\le & C \lambda^2 (\Delta t)^2 \int_{ t_{n-1}}^{ t_{n+1}} \|\partial_t \bld{\mathcal{U}}(\theta)\|_{L^2(\Sigma)}^2.
\end{alignat*}
\end{proof}

The error estimates are given in the following theorem. Note that we will implicitly assume the regularity mentioned in Remark \ref{remark}. In addition, we assume that 
$\partial_t \bld{\mathcal{U}} \in L^2(0,T;L^{2}(\Sigma))$ and $\partial_t \sigma_{\mathcal{F}}\bn \in L^2(0,T;L^{2}(\Sigma))$.
\begin{theorem}
Let $\bld{\mathcal{U}}, \mathcal{P}, \sigma_{\mathcal{F}}, \bld{\mathcal{E}}, \sigma_{\mathcal{S}}$ solve \eqref{Fluid}, \eqref{Structure} and \eqref{Exactinterface}. Furthermore, let $\bu^{n+1},  \sigma_{f}^{n+1}, p^{n+1}$ solve \eqref{fluid} and  $\bfeta^{n+1}, \dbfeta^{n+1}$ solve \eqref{solid}. If $T=N\Delta t $  with $N \ge 1$, the following estimate holds: 
\begin{alignat*}{1}
\mathbb{E}^{N} \ + \ \sum_{n=1}^{N}\mathbb{T}^{n} \ + \mathbb{S}^{N} \le  C \Delta t T \bigg(\lambda \| \partial_t \bld{\mathcal{U}}\|_{L^2(0,T;L^{2}(\Sigma))}^2 +\frac{1}{\lambda}\| \partial_t \sigma_{\mathcal{F}}\bn\|_{L^2(0,T;L^{2}(\Sigma))}^2 \bigg).
\end{alignat*}
\end{theorem}

\begin{proof}
 Using \eqref{solid6}, \eqref{fluid6} and \eqref{Exactinterface} we see that
\begin{align*}
\es^{n+1} \bn_{s} \ + \ \lambda \edeta^{n+1} \ &= \ \lambda \teu^{n} \ - \ \tef^{n} \bn \ + \ g_1^{n+1}, \\
\ef^{n+1} \bn \ + \ \lambda \eu^{n+1} \ &= \ \lambda \edeta^{n+1} \ + \ \tef^{n} \bn \ + \ g_2^{n+1},
\end{align*}
where $g_2^{n+1}$ is given in \eqref{g2} and
\begin{alignat*}{1}
g_1^{n+1} &:=\lambda ( \bld{\mathcal{U}}^{n+1}- \tilde{\bld{\mathcal{U}}}^{n}) \ + \ ( \tilde{\sigma}_{\mathcal{F}}^{n} \bn \ - \ \sigma_{\mathcal{F}}^{n+1}\bn). 
\end{alignat*}
By adding the two equations we get
\begin{subequations}\label{errorInterface} 
\begin{alignat}{1}
\es^{n+1} \bn_{s}+\ef^{n+1} \bn&= \ \lambda (\teu^{n}-\eu^{n+1} )+g_3^{n+1} \label{errorInterface1},
 \end{alignat}
 where $g_3^{n+1}$ is given in \eqref{g3}.  Also, we re-arrange the second equation and write 
\begin{alignat}{1}
 \eu^{n+1}-\edeta^{n+1} &= \frac{1}{\lambda} ( \tef^{n} \bn-\ef^{n+1} \bn) + \frac{1}{\lambda} g_{2}^{n+1}. \label{errorInterface2}
\end{alignat}
\end{subequations}

We may therefore proceed with the same initial steps from the stability analysis. This yields 
\begin{align*}
\frac{\rho_{f}}{2}\partial_{t} \|\eu^{n+1}\|^{2}_{L^{2}(\Omega_{f})} \ + \ 2 \mu \| {\eps}(\eu^{n+1}) \|^{2}_{L^{2}(\Omega_{f})} \ - \ \int_{\Sigma}\ef^{n+1}\bn \cdot \eu^{n+1} \ &= \ 0, \\
\frac{\rho_{s}}{2} \partial_{t}\|\edeta^{n+1}\|^{2}_{L^{2}(\Omega_{s})} \ + \ \frac{1}{2}\partial_{t}\|\eeta^{n+1}\|_S^2  \ - \ \int_{\Sigma}\es^{n+1} \bn_{s} \cdot \edeta^{n+1} \ &= \ 0.
\end{align*}
If we set 
\begin{equation*}
I^{n+1} := \frac{\rho_{f}}{2}\partial_{t} \|\eu^{n+1}\|^{2}_{L^{2}(\Omega_{f})} + 2 \mu \| {\eps}(\eu^{n+1}) \|^{2}_{L^{2}(\Omega_{f})} + \frac{\rho_{s}}{2} \partial_{t}\|\edeta^{n+1}\|^{2}_{L^{2}(\Omega_{s})} +\frac{1}{2}\partial_{t} \| \eeta^{n+1}\|_S^2,
\end{equation*}
we have that
\begin{alignat*}{1}
I^{n+1} =& \ \int_{\Sigma}\ef^{n+1}\bn \cdot \eu^{n+1} \ + \ \int_{\Sigma}\es^{n+1} \bn_{s} \cdot \edeta^{n+1} \\
=& \ \int_{\Sigma}\ef^{n+1} \bn \cdot (\eu^{n+1}- \edeta^{n+1}) \ + \ \int_{\Sigma}(\es^{n+1} \bn_{s}+\ef^{n+1} \bn) \cdot \edeta^{n+1} \\
=&  \ \frac{1}{\lambda} \int_{\Sigma}\ef^{n+1} \bn \cdot  ( \tef^{n} \bn-\ef^{n+1} \bn)+  \lambda\int_{\Sigma}(\teu^{n}-\eu^{n+1}) \cdot \edeta^{n+1} \\
&+  \ \frac{1}{\lambda} \int_{\Sigma}\ef^{n+1} \bn \cdot   g_2^{n+1}+ \int_{\Sigma}g_3^{n+1}\cdot \edeta^{n+1}.
\end{alignat*}
In the last equality we used \eqref{errorInterface}.  Also, the following holds after using \eqref{errorInterface}
\begin{alignat*}{1}
\|\tef^n\bn-\ef^{n+1} \bn\|_{L^{2}(\Sigma)}^2= & \|\lambda ( \eu^{n+1}-\edeta^{n+1})-g_2^{n+1} \|_{L^{2}(\Sigma)}^2\\
=& \lambda^2 \|\eu^{n+1}-\edeta^{n+1}\|_{L^{2}(\Sigma)}^2+ \|g_2^{n+1} \|_{L^{2}(\Sigma)}^2-2 \lambda \int_{\Sigma} g_2^{n+1} \cdot (\eu^{n+1}-\edeta^{n+1}). 
\end{alignat*}
If we use the above equations and \eqref{eq131} we obtain 
\begin{alignat*}{1}
I^{n+1} = &A + \frac{1}{\lambda} \int_{\Sigma}\ef^{n+1} \bn \cdot   g_2^{n+1}  +\int_{\Sigma}g_3^{n+1}\cdot \edeta^{n+1}
 - \frac{1}{2\lambda}\|g_{2}^{n+1}\|^{2}_{L^{2}(\Sigma)} + \int_{\Sigma} g_2^{n+1} \cdot (\eu^{n+1}-\edeta^{n+1}),
\end{alignat*}
where
\begin{equation*}
A := \frac{1}{2\lambda}\left(\|\tef^{n}\bn\|^{2}_{L^{2}(\Sigma)} - \|\ef^{n+1}\bn\|^{2}_{L^{2}(\Sigma)}\right) + \frac{\lambda}{2}\left(\|\teu^{n}\|^{2}_{L^{2}(\Sigma)} - \|\eu^{n+1}\|^{2}_{L^{2}(\Sigma)} - \|\edeta^{n+1} - \teu^{n}\|^{2}_{L^{2}(\Sigma)}\right).
\end{equation*}

Again using \eqref{errorInterface2} and applying Cauchy-Schwarz and Young's inequalities, we have 
\begin{alignat*}{1}
I^{n+1} = A&+\int_{\Sigma}g_3^{n+1}\cdot (\edeta^{n+1}-\teu^n)+\int_{\Sigma}g_3^{n+1}\cdot \teu^n+ \frac{1}{\lambda} \int_{\Sigma}\tef^{n} \bn \cdot   g_2^{n+1}   +\frac{1}{2\lambda}\|g_{2}^{n+1}\|^{2}_{L^{2}(\Sigma)} \\
\leq A &+ \frac{\lambda}{4}\|\edeta^{n+1} - \teu^{n}\|^{2}_{L^{2}(\Sigma)} + \frac{1}{\lambda}(1 + \frac{1}{2\delta})\|g_3^{n+1}\|_{L^{2}(\Sigma)}^{2} + \frac{1}{\lambda}(\frac{1}{2} + \frac{1}{2\delta})\|g_2^{n+1}\|_{L^{2}(\Sigma)}^{2}\\
& + \frac{\delta}{2\lambda}\|\tef^{n}\bn\|_{L^{2}(\Sigma)}^{2} + \frac{\lambda \delta}{2}\|\teu^n\|_{L^{2}(\Sigma)}^{2},
\end{alignat*}
where $\delta>0$.

Taking the integral on $[t_{n},t_{n+1}]$ and applying \eqref{eq111} we  have:
\begin{alignat*}{1}
\mathbb{E}^{n+1}+ \mathbb{T}^{n+1}  + \mathbb{S}^{n+1}  \le \mathbb{E}^{n}+  (1+\delta) \mathbb{S}^{n} +G^{n+1},
\end{alignat*}
where
\begin{alignat*}{2}
G^{n+1} \ &:= \  \frac{1}{\lambda}(1 + \frac{1}{2\delta})\int_{t_n}^{t_{n+1}}\|g_3^{n+1}(s)\|_{L^{2}(\Sigma)}^{2}ds + \frac{1}{\lambda}(\frac{1}{2} + \frac{1}{2\delta})\int_{t_n}^{t_{n+1}}\|g_2^{n+1}(s)\|_{L^{2}(\Sigma)}^{2}ds.
\end{alignat*}
We then clearly have 
\begin{equation*}
\mathbb{E}^{n+1}  +  \mathbb{T}^{n+1}  + \mathbb{S}^{n+1}  \leq  \mathbb{E}^{n}  +  \mathbb{S}^{n}  +\delta \max_{1 \leq m \leq N}\mathbb{S}^{m}  + G^{n+1}.
\end{equation*}
If we sum from $0$ to $M-1$ with $1\le M \le N$ and set $\delta = \frac{\Delta t}{2 T}$ we obtain
\begin{equation*}
\mathbb{E}^{M} \ + \ \sum_{n=0}^{M-1}\mathbb{T}^{n+1} \ + \ \mathbb{S}^{M} \ \leq   \frac{1}{2} \max_{1\leq m \leq N}\mathbb{S}^{m} + \sum_{n=0}^{M-1}G^{n+1}.
\end{equation*}
Here we used that $\mathbb{E}^{0}=0$ and $\mathbb{S}^{0}=0$. Since this holds for any $1 \le M \le N$ we have 
\begin{equation*}
\frac{1}{2} \max_{1\leq m \leq N}\mathbb{S}^{m} \le  \sum_{n=0}^{N-1}G^{n+1}.
\end{equation*}
Thus, we have 
\begin{equation}\label{e861}
\mathbb{E}^{N} \ + \ \sum_{n=1}^{N}\mathbb{T}^{n} \ + \ \mathbb{S}^{N} \ \leq  2 \sum_{n=0}^{N-1}G^{n+1}.
\end{equation}

Using  Lemma \ref{lem:g-bound} and  that $\delta = \frac{\Delta t}{2 T}$  we immediately have 
\begin{alignat*}{1}
\sum_{n=0}^{N-1} G^{n+1} & \leq C \Delta t T \bigg( \lambda \| \partial_t \bld{\UU}\|_{L^{2}(0,T; L^2(\Sigma))}^2 + \frac{1}{\lambda} \| \partial_t \sigma_{\mathcal{F}}\bn\|_{L^{2}(0,T; L^2(\Sigma))}^2 \bigg).
\end{alignat*}
Here we used that $T \ge \Delta t$. Combing this with \eqref{e861} completes the proof. 
\end{proof}


\section{Conclusion}
In this paper we analyzed a Robin-Robin splitting scheme for an FSI problem. We showed that the error is bounded by $\sqrt{T\Delta t}$.  Since the splitting does not discretize in space this gives several possibilities for spatial discretizations. In a forthcoming paper, we will analyze a fully discrete scheme and present numerical experiments.

\bibliographystyle{abbrv}
\bibliography{references}

\end{document}